\documentclass[12pt,twoside]{amsart}
\usepackage[margin=3cm]{geometry}
\usepackage[colorlinks=false]{hyperref}
\usepackage[english]{babel}
\usepackage{graphicx,titling}
\usepackage{float}
\usepackage{amsmath,amsfonts,amssymb,amsthm}
\usepackage{lipsum}
\usepackage[T1]{fontenc}
\usepackage{fourier}
\usepackage{color}
\usepackage[latin1]{inputenc}
\usepackage{esint}
\usepackage{caption}
\usepackage{ccicons}

\makeatletter
\def\blfootnote{\xdef\@thefnmark{}\@footnotetext}
\makeatother

\newcommand\ccnote{
    \blfootnote{\copyright\,\, Otis Chodosh, Christos Mantoulidis, and Felix Schulze}
    \blfootnote{\ccLogo\, \ccAttribution\,\, Licensed under a \href{https://creativecommons.org/licenses/by/4.0/}{Creative Commons Attribution License (CC-BY)}.}
}

\usepackage[export]{adjustbox}
\numberwithin{equation}{section}
\usepackage{setspace}

\renewcommand{\leq}{\leqslant}

\renewcommand{\geq}{\geqslant}
\renewcommand{\mathbb}{\varmathbb}
\usepackage{fancyhdr}
\newtheorem{theorem}{Theorem}[section]
\newtheorem{lemma}[theorem]{Lemma}
\newtheorem{corollary}[theorem]{Corollary}

\newtheorem{definition}[theorem]{Definition}
\newtheorem{remark}[theorem]{Remark}
\usepackage{mathrsfs}

\newtheorem{claim}[theorem]{Claim}

\newcommand{\NN}{\mathbb{N}}

\newcommand{\RR}{\mathbb{R}}

\newcommand{\cC}{\mathcal C}

\newcommand{\cH}{\mathcal H}

\newcommand{\cR}{\mathcal R}
\newcommand{\cS}{\mathcal S}

\newcommand{\bx}{\mathbf{x}}
\newcommand{\by}{\mathbf{y}}

\newcommand{\bOh}{\mathbf{0}}

\newcommand{\eps}{\varepsilon}

\DeclareMathOperator{\sing}{sing}
\DeclareMathOperator{\reg}{reg}

\DeclareMathOperator{\spt}{spt}

\DeclareMathOperator{\Graph}{graph}

\DeclareMathOperator{\Div}{div}

\DeclareMathOperator{\diam}{diam}
\newcommand{\mres}{\lfloor}

\address{Otis Chodosh, Stanford University, Department of Mathematics, Bldg.\ 380, Stanford, CA 94305, USA}
\email{ochodosh@stanford.edu}
\medskip
\address{Christos Mantoulidis, Rice University, Department of Mathematics, Herman Brown Hall, Houston, TX 77005, USA} 
\email{christos.mantoulidis@rice.edu}
\medskip
\address{Felix Schulze, University of Warwick, Department of Mathematics, Zeeman Building, Gibbet Hill Road, Coventry CV4 7AL, UK}
\email{felix.schulze@warwick.ac.uk}

\begin{document}

\thispagestyle{empty}

\begin{minipage}{0.28\textwidth}
\begin{figure}[H]
\includegraphics[width=2.5cm,height=2.5cm,left]{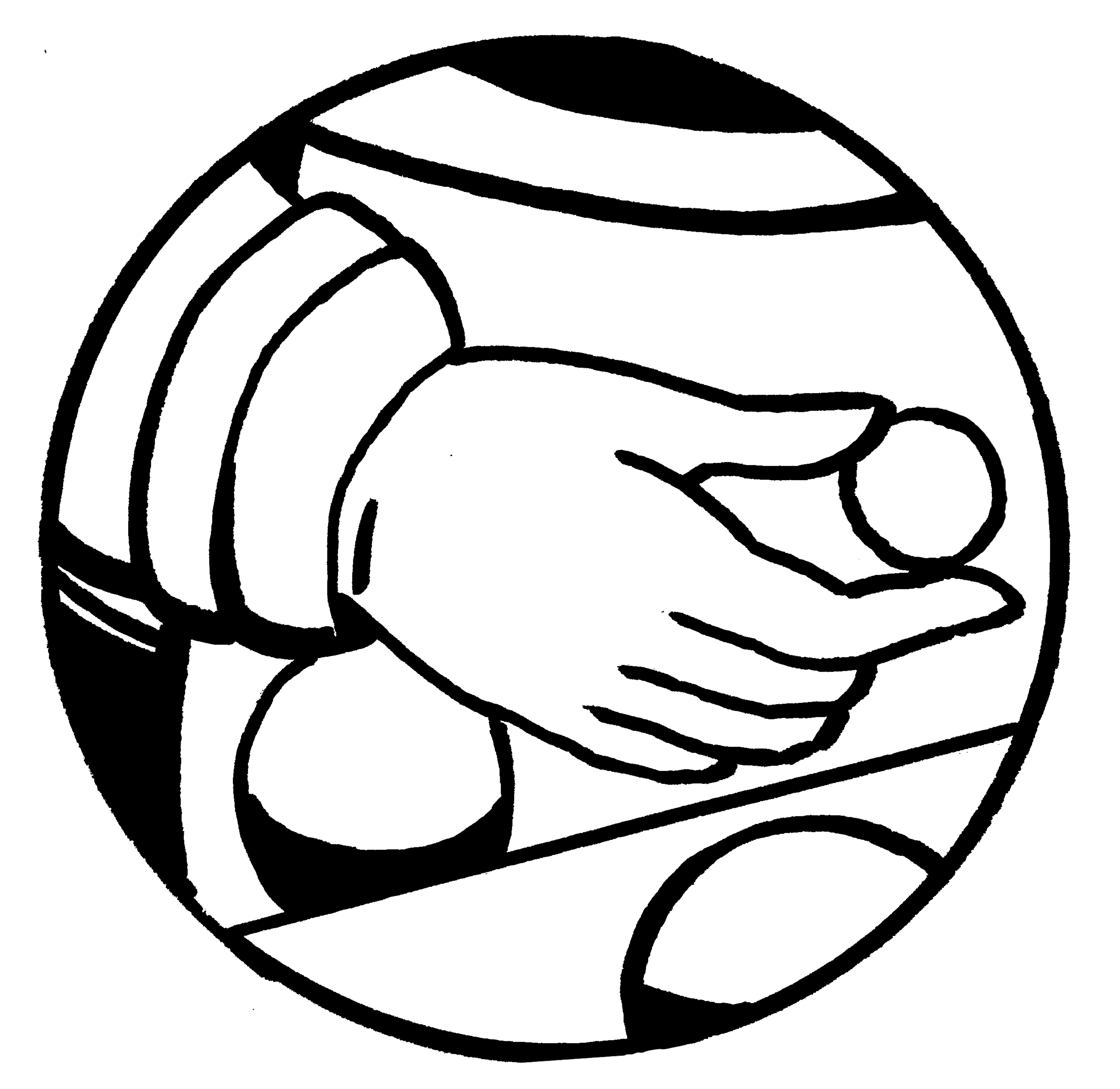}
\end{figure}
\end{minipage}
\begin{minipage}{0.7\textwidth} 
\begin{flushright}
Ars Inveniendi Analytica (2024), Paper No. 3, 16 pp.
\\
DOI 10.15781/z70n-ed29
\\
ISSN: 2769-8505
\end{flushright}
\end{minipage}

\ccnote

\vspace{1cm}


\begin{center}
\begin{huge}
\textit{Improved generic regularity of codimension-1 minimizing integral currents}


\end{huge}
\end{center}

\vspace{1cm}


\begin{minipage}[t]{.28\textwidth}
\begin{center}
{\large{\bf{Otis Chodosh}}} \\
\vskip0.15cm
\footnotesize{Stanford University}
\end{center}
\end{minipage}
\hfill
\noindent
\begin{minipage}[t]{.28\textwidth}
\begin{center}
{\large{\bf{Christos Mantoulidis}}} \\
\vskip0.15cm
\footnotesize{Rice University}
\end{center}
\end{minipage}
\hfill
\noindent
\begin{minipage}[t]{.28\textwidth}
\begin{center}
{\large{\bf{Felix Schulze}}} \\
\vskip0.15cm
\footnotesize{University of Warwick} 
\end{center}
\end{minipage}

\vspace{1cm}


\begin{center}
\noindent \em{Communicated by Joaquim Serra}
\end{center}
\vspace{1cm}


\noindent \textbf{Abstract.} \textit{Let $\Gamma$ be a smooth, closed, oriented, $(n-1)$-dimensional submanifold of $\RR^{n+1}$. We show that there exist arbitrarily small perturbations $\Gamma'$ of $\Gamma$ with the property that minimizing integral $n$-currents with boundary $\Gamma'$ are smooth away from a set of Hausdorff dimension $\leq n-9-\eps_n$, where $\eps_n \in (0, 1]$ is a dimensional constant. \\ \\%
%
This improves on our previous result (where we proved generic smoothness of minimizers in $9$ and $10$ ambient dimensions). The key ingredients developed here are a new method to estimate the full singular set of the foliation by minimizers and a proof of superlinear decay of closeness (near singular points) that holds even across non-conical scales.}
\vskip0.3cm

\vspace{0.5cm}


\section{Introduction}

In \cite{CMS:generically.smooth.10} we showed that the smooth, oriented area minimization problem is generically solvable up to ambient dimension 10:

\begin{theorem} \label{theo:prev}
	Let $n+1 \in \{ 8, 9, 10 \}$ and $\Gamma \subset \RR^{n+1}$ be a smooth, closed, oriented, $(n-1)$-dimensional submanifold of $\RR^{n+1}$. There exist arbitrarily small perturbations $\Gamma'$ of $\Gamma$ (as $C^\infty$ graphs in the normal bundle of $\Gamma$) with the property that there exists a least-area smooth, compact, oriented hypersurface $M' \subset \RR^{n+1}$ with $\partial M' = \Gamma'$. 
\end{theorem}

In this paper, we prove the following sharper geometric measure theory result in all dimensions, which implies the above theorem when $n+1 \in \{ 8, 9, 10 \}$. We will implicitly assume $n+1\geq 8$ throughout the paper, since otherwise there is nothing to show.

\begin{theorem} \label{theo:main.geometric}
	Let $\Gamma$ be a smooth, closed, oriented, $(n-1)$-dimensional submanifold of $\RR^{n+1}$. There exist arbitrarily small perturbations $\Gamma'$ of $\Gamma$ (as $C^\infty$ graphs in the normal bundle of $\Gamma$) such that every minimizing integral $n$-current with boundary $\llbracket \Gamma' \rrbracket$ is of the form $\llbracket M' \rrbracket$ for a smooth, precompact, oriented hypersurface $M'$ with $\partial M' = \Gamma'$ and
	\[ \sing M' = \emptyset \text{ if $n+1 \leq 10$, else } \dim_H \sing M' \leq n-9-\eps_n \]
	where $\eps_n \in (0, 1]$ is the dimensional constant defined in \eqref{eq:eps.n}. In fact the singular strata $\cS^\ell(M')$, $\ell \in \NN$, of each such $M'$ (see Definition \ref{defi:strata}) can be arranged to satisfy 
	\[ \cS^0(M') = \cS^1(M') = \cS^2(M') = \emptyset, \; \dim_H \cS^\ell(M') \leq \ell - 2 - \eps_n \text{ for } \ell \geq 3, \]
	on top of the standard regularity $\cS^\ell(M') = \emptyset$ for $\ell > n-7$.
\end{theorem}

\begin{remark}
For example, when $n+1=11$ this shows that every minimizer $M'$ for $\Gamma'$ has 
\begin{equation}\label{eq:11-dim}
\cS^0(M')=\cS^1(M')=\cS^2(M')= \emptyset \textrm{\quad and \quad } \dim_H\cS^3(M') \leq 1-\eps_{10} \approx 0.65
\end{equation}
(see Remark \ref{rema:eps.n} below). This should be compared with the fact that $\cS^3(M')$ is $3$-rectifiable \cite{Simon:cyl,NaberValtorta}. Note that examples of \emph{stable} hypersurfaces having singular set satisfying \eqref{eq:11-dim} have been recently constructed in \cite{Simon:ex}. 
\end{remark}

The dimensional constant $\eps_n$ comes from the analysis of minimizing cones, and specifically relates to the rate of decay in the radial direction of positive Jacobi fields on $n$-dimensional minimizing cones, which can be bounded from above by the constant
\begin{equation} \label{eq:kappa.n}
	\kappa_n = \frac{n-2}{2} - \sqrt{\frac{(n-2)^2}{4} - (n-1)} \in (1, 2];
\end{equation}
see \cite{Simon:decay, Wang:smoothing.cones} and Lemma \ref{lemm:jacobi.field.growth}. Specifically, $\eps_n$ is given by:
\begin{equation} \label{eq:eps.n}
	\eps_n = \kappa_n - 1 \in (0, 1].
\end{equation}

\begin{remark} \label{rema:eps.n}
	A computation shows that $\eps_n$ decreases toward $0$, with initial values:
	\begin{align*}
		\eps_7 & = 1, \\
		\eps_8 & \approx 0.58, \\
		\eps_9 & \approx 0.44, \\
		\eps_{10} & \approx 0.35.
	\end{align*}
\end{remark}

Theorem \ref{theo:main.geometric} follows from the combination of two independent results about \textit{families} of minimizers. The first result is a bound on the size of the union of strata for a family of pairwise disjoint minimizers. Since it is local, we state it for minimizing boundaries inside open sets.

\begin{theorem} \label{theo:singular.set.size}
	Let $\mathscr{F}$ be a family of minimizing boundaries in an open set $U \subset \RR^{n+1}$ whose supports are pairwise disjoint in $U$. For $\ell \in \NN$, we have
	\[ \cS^\ell(\mathscr{F}) = \cup_{T \in \mathscr{F}} \cS^\ell(T) \implies  \dim_H \cS^\ell(\mathscr{F}) \leq \ell. \]
\end{theorem}

\begin{remark} \label{rema:singular.set.size}
	Note that:
	\begin{enumerate}
		\item[(a)] When $\ell=0$, the work of Hardt--Simon \cite{HardtSimon:isolated.singularities} implies that $\cS^0(\mathscr{F})$ is discrete. 
		\item[(b)] If $\mathscr{F}$ is a singleton, Theorem \ref{theo:singular.set.size} recovers the standard bound on the size of the strata of a single minimizer (see Remark \ref{rema:strata}).
	\end{enumerate}
\end{remark}

The second result proves, for families of pairwise disjoint minimizers with prescribed smooth boundaries, that if one minimizer is near the singular part of another then the closeness propagates to the boundary with a superlinear rate relating to $\kappa_n$ from \eqref{eq:kappa.n}. To state the result we need to consider for smooth, closed, oriented, $(n-1)$-dimensional $\Gamma \subseteq \RR^{n+1}$, the set of all possible minimizers with boundary $\Gamma$:
\[ \mathscr{M}(\Gamma) = \{ \text{minimizing integral $n$-currents } T \text{ in $\RR^{n+1}$ with } \partial T = \llbracket \Gamma \rrbracket \}. \]

\begin{theorem} \label{theo:timestamp.regularity}
	Let $(\Gamma_s)_{s \in [-\delta, \delta]}$ be a smooth deformation of $\Gamma_0 = \Gamma$, a smooth, closed, oriented, $(n-1)$-dimensional submanifold of $\RR^{n+1}$. Consider the family
	\[ \mathscr{F} = \cup_{s \in [-\delta, \delta]} \mathscr{M}(\Gamma_s). \]
	and assume the following:
	\begin{enumerate}
		\item[(a)] All elements of $\mathscr{F}$ with distinct boundaries have pairwise disjoint supports.
		\item[(b)] All elements of $\mathscr{F}$ have multiplicity-one up to their boundary.
		\item[(c)] All elements of $\mathscr{F}$ are near their boundary graphical over a fixed hypersurface $\Sigma$ with nonempty boundary; specifically, there exists $h : \mathscr{F} \to C^\infty(\Sigma)$ so that for all $s \in [-\delta, \delta]$, $T_s \in \mathscr{M}(\Gamma_s)$:
		\[ \Graph_{\Sigma} h(T_s) \subset \spt T_s, \; \partial (\Graph_{\Sigma} h(T_s)) = \Gamma_s. \]
		\item[(d)] The graph map $h : \mathscr{F} \to C^\infty(\Sigma)$ is increasing along $\Gamma$ with a definite rate $\alpha > 0$ in the sense that for all $s_j \in [-\delta, \delta]$, $T_{s_j} \in \mathscr{M}(\Gamma_{s_j})$, $j = 1, 2$,
		\[ s_1 < s_2 \implies h(T_{s_2}) - h(T_{s_1}) \geq \alpha (s_2 - s_1) \text{ on } \Gamma. \]
	\end{enumerate}
	For convenience, denote
	\[ \spt \mathscr{F} = \cup_{s \in [-\delta, \delta]} \cup_{T_s \in \mathscr{M}(\Gamma_s)} \spt T_s, \]
	\[ \sing \mathscr{F} = \cup_{s \in [-\delta, \delta]} \cup_{T_s \in \mathscr{M}(\Gamma_s)} \sing T_s. \]
	Then, the timestamp function
	\[ \mathfrak{t} : \spt \mathscr{F} \to [-\delta, \delta], \]
	\[ \mathfrak{t}(\bx) = s \text{ for all } \bx \in \spt T_s, \; T_s \in \mathscr{M}(\Gamma_s), \; s \in [-\delta, \delta], \]
	is $\alpha$-H\"older on $\sing \mathscr{F}$ for every $\alpha \in (0, \kappa_n + 1)$ with $\kappa_n$ as in \eqref{eq:kappa.n}.
\end{theorem}

Theorems \ref{theo:singular.set.size} and \ref{theo:timestamp.regularity} imply:

\begin{corollary} \label{coro:main}
	Let $(\Gamma_s)_{s \in [-\delta, \delta]}, (\mathscr{M}(\Gamma_s))_{s \in [-\delta, \delta]}$ be as in Theorem \ref{theo:timestamp.regularity}. Then, 
	\[ \cS^0(T_s) = \cS^1(T_s) = \cS^2(T_s) = \emptyset, \; \dim_H \cS^\ell(T_s) \leq \ell - 9 - \eps_n \text{ for } \ell \geq 3, \]
	for all $T_s \in \mathfrak{M}(\Gamma_s)$ for a.e.\ $s \in [-\delta, \delta]$, where $\eps_n > 0$ is as in \eqref{eq:eps.n}.
\end{corollary}

All these tools can be put together to yield Theorem \ref{theo:main.geometric}.

\begin{remark}
	These same improved regularity results should hold for homological minimizers in Riemannian manifolds under generic perturbations of the metric, similarly to  \cite{CMS:generically.smooth.10}. 
\end{remark}

\begin{remark}
 \label{rema:FROS}
As already pointed out in our previous work \cite{CMS:generically.smooth.10}, there is a connection to the recent work of Figalli--Ros-Oton--Serra \cite{FigalliRosOtonSerra:obstacle.generic.regularity} 
on generic regularity for free boundaries in the obstacle problem. That work, too, relies on a subtle derivation of superlinear H\"older-continuity estimate on a timestamp function for a foliation to prove the smallness of a spacetime singular set across all time parameters $t$. In our previous work \cite{CMS:generically.smooth.10}, both of these tools were \emph{coupled} with a \emph{maximal density drop} argument. This prevented us from estimating the singular set in high dimensions (as we do here) since it was hard to iterate the estimate in that form. Here, we develop new techniques that allow us to iterate the density drop argument at an earlier stage. This then allows us to obtain stronger results (analogous to the full dimensional range of  \cite{FigalliRosOtonSerra:obstacle.generic.regularity}). 
\end{remark}

\subsection{Organization} 
Section \ref{sec:defi} contains the basic definitions. In Section \ref{sec:singular.set.size} we estimate the dimension of the foliation singular strata and in Section \ref{theo:timestamp.regularity} we prove the super-linear separation estimates (even across non-conical scales). In Section \ref{sec:proof-main-cor} we combine these pieces to estimate the size of the singular strata of generic minimizers. Finally in Section \ref{theo:construct-foliation} we construct the foliations to which the previous results apply. 

\subsection{Acknowledgements} 
O.C. was supported by a Terman Fellowship and an NSF grant (DMS-2304432). C.M. was supported by an NSF grant (DMS-2147521). We are grateful to the referee for their careful reading and helpful suggestions. 

\section{Definitions}\label{sec:defi}

Let us collect the definitions we are going to use. Below, $U \subset \RR^{n+1}$ is open and $T$ is any minimizing integral $n$-current in $U$ (see \cite[\S 33]{Simon:GMT}, with $A=U$). 

\begin{remark} \label{rema:hypersurface.notation}
For notational simplicity, for minimizing integral $n$-currents $T$ of the form $\llbracket M \rrbracket$ for a smooth hypersurface $M$ with or without boundary we will use the definitions below with $M$ instead of with $\llbracket M \rrbracket$.
\end{remark}

\begin{definition} \label{defi:reg.sing}
	We denote 
	\begin{align*}
		\reg T = \{ \bx \in U \cap \spt T \setminus \spt \partial T : &\;  \spt T \cap B_r(\bx) \text{ is a smooth hypersurface} \\
			& \text{without boundary for some } r > 0 \},
	\end{align*}
	and
	\[ \sing T = U \cap \spt T \setminus (\spt \partial T \cup \reg T). \]
\end{definition}

In Section \ref{sec:timestamp.regularity} we will want to study subsets of $\reg T$ with effective regularity:

\begin{definition} \label{defi:reg.scale}
	For $\bx \in \reg T$, we define the regularity scale at $\bx$, $r_T(\bx) \in (0, 1]$, as the supremum of $r \in (0, 1)$ so that $\partial T = 0$ in $B_r(\bx)$ and $T \mres B_r(\bx)$ is supported on a smooth hypersurface with second fundamental form $|A| \leq r^{-1}$. For all other $\bx \in \spt T$, we set $r_T(\bx) = 0$. We also denote, for $\delta > 0$, the following effective portion of $\reg T$:
	\[ \cR_{\geq \delta}(T) = \{ \bx \in \reg T : r_T(\bx) \geq \delta \}. \]
\end{definition}

\begin{remark} \label{remark:reg.scale}
	One can show (\cite[Lemma 2.4]{CMS:generically.smooth.10}) that $r_T(\bx)$ is continuous in both $\bx$ and $T$, provided $T$ varies among minimizing integral $n$-currents with the flat distance and the Hausdorff distance on their boundaries (if the boundaries are nontrivial).
\end{remark}

In Theorem \ref{theo:singular.set.size} we will want to study refined subsets of $\sing T$ called singular strata. Note that $\sing T \subset \spt T \setminus \spt \partial T$ in Definition \ref{defi:reg.sing}. Since minimizing integral $n$-currents $T$ decompose locally away from $\spt \partial T$ into sums of integer multiples of minimizing boundaries (by \cite[\S 27]{Simon:GMT}) with pairwise disjoint supports (by \cite{Simon:smp}), in the rest of this section we take $T$ to be a minimizing \textit{boundary} in $U$ (see \cite[\S 37]{Simon:GMT}). For all other $T$, one combines the definitions by taking unions over all balls away from $\spt \partial T$.

It is well-known (see \cite[\S 35]{Simon:GMT}) that, when $T$ is a minimizing boundary, blow-ups at $\bx \in \sing T$ are $n$-dimensional minimizing \textit{cones} $\cC \subset \RR^{n+1}$.

\begin{definition} \label{defi:spine}
	The spine of a cone $\cC \subset \RR^{n+1}$ is the largest subspace $\Pi \subset \RR^{n+1}$ such that $\cC = \Pi \times \cC_0$ for a cone $\cC_0 \subset \RR^{n+1-k}$, $k = \dim \operatorname{spine} \cC$. Equivalently, $\Pi$ is the set of points under which $\cC$ is invariant by translation (see \cite[\S 3]{White:stratification}).
\end{definition}

\begin{definition} \label{defi:strata}
	For each $\ell \in \NN$, we define the $\ell$-th singular stratum of $T$ to be
	\[ \cS^\ell(T) = \{ \bx \in \sing T : \dim \operatorname{spine} \cC \leq \ell \text{ for all tangent cones } \cC \text{ of } T \text{ at } \bx \}. \]
\end{definition}

\begin{remark} \label{rema:strata}
	It is well-known (cf. \cite[\S 4]{White:stratification}) that
	\begin{enumerate}
		\item[(a)] $\dim_H \cS^\ell(T) \leq \ell$ for all $\ell \in \NN$, 
		\item[(b)] $\cS^0(T)$ is discrete, and
		\item[(c)] $\cS^\ell(T) = \emptyset$ for $\ell > n-7$.
	\end{enumerate}
	Note that (a) and (c) together imply the celebrated result that $\dim_H \sing T \leq n-7$.
\end{remark}

We will also need to study more effective subsets of the singular strata:

\begin{definition} \label{defi:effective.strata}
	For $\ell \in \NN$, $\eps > 0$, we also set
	\[ \cS^\ell_\eps(T) = \{ \bx \in \sing T : \text{all tangent cones } \cC \text{ of } T \text{ at } \bx \text{ are} \geq \eps \text{ from splitting an } \RR^{\ell+1} \}. \]
	That is,  $\bx \in \cS^\ell_\eps(T)$ if $\bx \in \sing T$ and each tangent cone $\cC$ of $T$ at $\bx$ satisfies
	\[ d_{B_1(\bOh)}(\cC, \llbracket \Pi \rrbracket \times \cC_0) \geq \eps \]
	for all $(\ell+1)$-dimensional subspaces $\Pi \subset \RR^{n+1}$ and all minimizing cones $\cC_0 \subset \RR^{n-\ell}$; here, $d_{B_1(\bOh)}$ denotes the flat metric for integral $n$-currents in $B_1(\bOh)$ (see \cite[\S 31]{Simon:GMT}).
\end{definition}

This definition is inspired by the quantitative strata defined by Cheeger--Naber \cite{CheegerNaber:harmonic.minimal} but is a distinct notion: we are only studying the symmetries at the tangent cone level, i.e., after blowing up, whereas the quantitative strata of Cheeger--Naber study the symmetries on intervals of scales before any blow-ups.

\begin{remark} \label{rema:effective.strata}
	Note that:
	\begin{enumerate}
		\item[(a)] $\cS^\ell_{\eps_2}(T) \subset \cS^\ell_{\eps_1}(T) \subset \cS^\ell(T)$ for all $0 < \eps_1 < \eps_2$, and
		\item[(b)] $\cS^\ell(T) = \cup_{\eps > 0} \cS^\ell_\eps(T)$.
	\end{enumerate}
\end{remark}

Finally, we will also need the following definition:

\begin{definition} \label{defi:cross.smoothly}
We say that $T$ and $T'$  \textit{cross smoothly} at $p \in \reg T \cap \reg T'$ if, for all sufficiently small $r > 0$, there are points of $\reg T'$ on both sides of $\reg T$ within $B_r(p)$ and vice versa; that is, for small enough $r > 0$ that $B_r(p) \setminus \reg T$ and $B_r(p) \setminus \reg T'$ each consist of pairs of components $U_\pm$ and $U_\pm'$, respectively, then the sets$$\reg T \cap U_+', \; \reg T \cap U_-', \; \reg T' \cap U_+, \; \reg T' \cap U_-$$ are all nonempty.
\end{definition}

\section{Proof of Theorem \ref{theo:singular.set.size}} \label{sec:singular.set.size}

\begin{lemma}\label{lemm:cone.split.pre}
Let $\gamma > 0$ and $\eps > 0$ be given. There exists $\eta = \eta(n, \gamma, \eps) \in (0, 1)$ with the following property. 

Consider any minimizing cone $\cC$ in $\RR^{n+1}$ with $\dim \operatorname{spine} \cC \leq \ell$, and that $\cC$ is $\geq \eps$ from splitting an $\RR^{\ell+1}$. Let $\cS \subset \bar B_1(\bOh)$ be the set of all points $\bx \in \bar B_1(\bOh) \cap \sing T$, where $T$ is any minimizing boundary in $\RR^{n+1}$ that does not cross $\cC$ smoothly, and where
\[
\Theta_{T}(\bx) \geq \Theta_\cC(\bOh) - \eta.
\]
Then, $\cS \subset U_\gamma(\Pi)$ for some $\leq \ell$-dimensional subspace $\Pi\subset \RR^{n+1}$.
\end{lemma}
\begin{proof}
	Suppose, for contradiction, that this failed with $\eta = j^{-1}$, $j = 2, 3, \ldots$ and cones $\cC_j$. Passing to a subsequence (not labeled), we can assume that $\cC_j \to \cC$. Since $\cC$ is $\geq \eps$ from splitting an $\RR^{\ell+1}$, we have that $\dim \Pi \leq \ell$ for $\Pi : = \operatorname{spine} \cC$. 
	
	The contradiction hypothesis guarantees that for each $j$ there exist $T_j$, $\bx_j$ as above, with
	\begin{equation} \label{eq:cone.split.pre.contradiction.1}
		\Theta_{T_j}(\bx_j) \geq \Theta_{\cC_j}(\bOh) - j^{-1},
	\end{equation}
	\begin{equation} \label{eq:cone.split.pre.contradiction.2}
		\bx_j \not \in U_\gamma(\Pi). 
	\end{equation}
	By \eqref{eq:cone.split.pre.contradiction.1} and the upper semicontinuity of density, we have $\Theta_T(\bx) \geq \Theta_\cC(\bOh)$. Since $T$ does not smoothly cross $\cC$ (or else some $T_j$ would), \cite[Proposition 3.3]{CMS:generically.smooth.10} implies $T = \cC$ and $\bx \in \Pi$. On the other hand, by \eqref{eq:cone.split.pre.contradiction.2} implies that $\bx \not \in U_\gamma(\Pi)$, a contradiction.
\end{proof}

\begin{proof}[Proof of Theorem \ref{theo:singular.set.size}]
	We may suppose that $U = B_1(\bOh)$ since Hausdorff dimension upper bounds are preserved under countable unions and scaling.
	
	The theorem will follow if we show that, for all $\delta > 0$,
	\begin{equation} \label{eq:singular.set.size.reduction.delta}
		\cH^{\ell+\delta}(\cS^\ell(\mathscr{F})) = 0.
	\end{equation}
	So let's fix $\delta > 0$ going forward. 
	
	We will need the $\infty$-approximation to Hausdorff measure $\cH^{d}$, for $d > 0$ real, denoted $\cH^{d}_\infty$. It is defined for all $A \subset \RR^{n+1}$ by $\cH^{d}_\infty(A) = \inf \{ \omega_{d} \sum_{j=1}^\infty (\tfrac12 \diam C_j)^{d} \}$, where the $\inf$ is taken among all covers $\{ C_j \}_{j=1,2,\ldots}$ of $A$ and $\omega_d$ is usually taken to be the volume of the unit $d$-ball when $d$ is an integer and its analytic extension to all $d > 0$ via the $\Gamma$ function, though the particular choice doesn't matter (see \cite[\S 2]{Simon:GMT}). 
	
	Note that if $\Pi \subset \RR^{n+1}$ is any subspace with $\dim \Pi \leq \ell$, then for $\gamma > 0$,
	\begin{equation} \label{eq:singular.set.size.growth}
		\cH^{\ell+\delta}_\infty(U_{2\gamma}(\Pi) \cap B_1(\bOh)) \leq C_{n,\ell,\delta} \gamma^\delta \text{ for all } \gamma > 0;
	\end{equation}
	this can be seen, e.g., by constructing an explicit covering of $U_{2\gamma}(\Pi) \cap B_1(\bOh)$. Now fix $\gamma \in (0, 1)$, depending only on $n$, $\ell$, $\delta$, so that
	\begin{equation} \label{eq:singular.set.size.gamma}
		C_{n,\ell,\delta} \gamma^\delta \leq \tfrac12 \cdot 2^{-\ell-\delta} \omega_{\ell+\delta}.
	\end{equation}
	
	Suppose, for contradiction, that \eqref{eq:singular.set.size.reduction.delta} fails. By Remark \ref{rema:effective.strata}, the set $$\cS^\ell_\eps(\mathscr{F}) = \cup_{T \in \mathscr{F}} \cS^\ell_\eps(T)$$ would then satisfy
	\begin{equation} \label{eq:singular.set.size.reduction.delta.eps}
		\cH^{\ell+\delta}(\cS^\ell_\eps(\mathscr{F})) > 0,
	\end{equation}
	for some $\eps > 0$, which we also fix. This now determines $\eta = \eta(n, \gamma, \eps)$ per Lemma \ref{lemm:cone.split.pre}. Using this $\eta$, define, for $k \in \NN$,
	\[
		\cS^{\ell,k}_{\eps}(\mathscr{F}) : = \cup_{T \in \mathscr{F}} \{ \by \in \cS^\ell_{\eps}(\mathscr{F}) \cap \sing T : 1+k\eta \leq \Theta_T(\by) < 1 + (k+1)\eta\}
	\]
	so that,
	\begin{equation} \label{eq:singular.set.size.decomposition}
		\cS^\ell_\eps(\mathscr{F}) = \cup_{k=0}^\infty \cS^{\ell,k}_{\eps}(\mathscr{F}).
	\end{equation}
	It follows from \eqref{eq:singular.set.size.reduction.delta.eps} and \eqref{eq:singular.set.size.decomposition} that, for some $k \in \NN$,
	\begin{equation} \label{eq:singular.set.size.reduction.delta.eps.k}
		\cH^{\ell+\delta}(\cS^{\ell,k}_{\eps}(\mathscr{F})) > 0.
	\end{equation}
	Since $U = B_1(\bOh)$ is bounded, \cite[3.6 (2)]{Simon:GMT} applies and guarantees that
	\[ \limsup_{\lambda \to 0} \frac{\cH^{\ell+\delta}_\infty(\cS^{\ell,k}_\eps(\mathscr{F}) \cap B_\lambda(\bx))}{\omega_{\ell+\delta} \lambda^{\ell+\delta}} \geq 2^{-\ell-\delta} \text{ for } \cH^{\ell+\delta} \text{ a.e.\ } \bx \in \cS^{\ell,k}_\eps(\mathscr{F}). \]
	Fix any $\bx$ as above. Then there is a sequence $\lambda_i \to 0$ such that
	\begin{equation} \label{eq:singular.set.size.measure}
		\lim_i \lambda_i^{-\ell-\delta} \cH^{\ell+\delta}_\infty(\cS^{\ell,k}_\eps(\mathscr{F}) \cap B_{\lambda_i}(\bx)) \geq 2^{-\ell-\delta} \omega_{\ell+\delta}.
	\end{equation}
	Since $\bx \in \cS^{\ell}_\eps(T)$ for some $T \in \mathscr{F}$, after passing to a subsequence (not labeled) we have
	\[ (\eta_{\bx, \lambda_i})_{\#} T \to \cC, \]
	for a minimizing cone $\cC$ that's $\geq \eps$ from splitting a $\RR^{\ell+1}$. Choose a $\leq \ell$-dimensional subspace $\Pi$ by applying Lemma \ref{lemm:cone.split.pre} to $\cC$ (with $\gamma,\eps$ as fixed above).
	
	We claim that, for $i$ sufficiently large,
	\begin{equation} \label{eq:singular.set.size.claim}
		\lambda_i^{-1} (\cS^{\ell,k}_\eps(\mathscr{F}) - \bx) \cap B_1(\bOh) \subset U_{2\gamma}(\Pi),
	\end{equation}
	Indeed, if we show this, then \eqref{eq:singular.set.size.growth} and \eqref{eq:singular.set.size.gamma} imply
	\[ \cH^{\ell+\delta}_\infty(\lambda_i^{-1} (\cS^{\ell,k}_\eps(\mathscr{F}) - \bx) \cap B_1(\bOh)) \leq \cH^{\ell+\delta}_\infty(U_{2\gamma}(\Pi) \cap B_1(\bOh)) \leq \tfrac12 \cdot 2^{-\ell-\delta} \omega_{\gamma+\delta}, \]
	in contradiction to \eqref{eq:singular.set.size.measure}. 
	
	It remains to verify \eqref{eq:singular.set.size.claim}. Suppose it failed with $i \to \infty$. Then, there would exist
	\begin{equation} \label{eq:choice-y-i-sing-set-size-claim}\by_i \in \lambda_i^{-1} (\cS^{\ell,k}_\eps(\mathscr{F}) - \bx) \cap B_1(\bOh) \setminus U_{2\gamma}(\Pi). \end{equation}
	By our definition of $\cS^{\ell,k}_\eps(\mathscr{F})$, we have $\by_i \in \sing (\eta_{\bx,\lambda_i})_\#T_i$ for some $T_i \in \mathscr{F}$ and
	\[ \Theta_{(\eta_{\bx,\lambda_i})_\# T_i}(\by_i) \geq 1 + k\eta \geq \Theta_\cC(0) - \eta. \]
	Passing to a subsequence we find $(\eta_{\bx,\lambda_i})_\# T_i\to T$ a minimizing boundary in $\RR^{n+1}$ which does not cross $\cC$ smoothly (otherwise, some $T_i$ would cross $T$ smoothly, which is impossible since elements of $\mathscr{F}$ have pairwise disjoint supports) and $\by_i\to \by$ with
	\[
	\Theta_T(\by) \geq 1+k\eta \geq \Theta_\cC(\bOh) - \eta.
	\] 
	By choice of $\Pi$ above---based on Lemma \ref{lemm:cone.split.pre}---we find that $\by \in U_{\gamma}(\Pi)$. This contradicts the choice of $\by_i$ in \eqref{eq:choice-y-i-sing-set-size-claim}. 
\end{proof}

\section{Proof of Theorem \ref{theo:timestamp.regularity}} \label{sec:timestamp.regularity}

\begin{lemma} \label{lemm:nontrivial.reg}
	There exists $\rho_n > 0$ with the following property.
	
	If $T$ is a minimizing boundary in $B_2(\bOh) \subset \RR^{n+1}$, and $\spt T \cap \bar B_{1/2}(\bOh) \neq \emptyset$, then
	\[ \cR_{\geq \rho_n}(T) \cap \partial B_1(\bOh) \neq \emptyset. \]
\end{lemma}
\begin{proof}
	Suppose, for contradiction, that for each $j = 1, 2, \ldots$ we could find $T_j$ as above, except with
	\begin{equation} \label{eq:nontrivial.reg.contradiction}
		\cR_{\geq 1/j}(T_j) \cap \partial B_1(\bOh) = \emptyset.
	\end{equation}
	We can pass to a subsequence (not denoted) along which $T_j \to T$, a minimizing boundary with
	\[
	\spt T \cap \bar B_{1/2}(\bOh) \neq \emptyset\,.
	\]
	Note that $\spt T \cap \partial B_1(\bOh) \neq \emptyset$ by monotonicity, and thus $\cup_{j = 1, 2, \ldots} \cR_{\geq 1/j}(T) \cap \partial B_1(\bOh) \neq \emptyset$ since $\dim_H \sing T \leq n-7$. In particular, we must have $\cR_{1/j}(T) \cap \partial B_1(\bOh) \neq \emptyset$ for some $j$, contradicting \eqref{eq:nontrivial.reg.contradiction} and Remark \ref{remark:reg.scale}.
\end{proof}

\begin{lemma} \label{lemm:separation.control}
	Let $A > 1$, $\rho \in (0, \rho_n)$ be given, with $\rho_n$ as in Lemma \ref{lemm:nontrivial.reg}. There exists $L = L(n, A, \rho)$ with the following properties.
	
	Take a minimizing boundary $T$ in $B_{2A}(\bOh)$ with $\bOh \in \spt T$. Then, for every minimizing boundary $T'$ in $B_{2A}(\bOh)$ not crossing $T$ smoothly, and with $\spt T' \cap \bar B_{1/2}(\bOh) \neq \emptyset$,
	\[ d(\cR_{\geq A \rho}(T) \cap \partial B_A(\bOh), \spt T') \leq L \cdot d(\cR_{\geq \rho}(T) \cap \partial B_1(\bOh), \spt T'). \]
\end{lemma}
\begin{proof}
	Without loss of generality, we may suppose $T$, $T'$ have connected supports. 
	
	Suppose, for contradiction, that for each $j = 1, 2, \ldots$ we could find $T_j, T_j'$ as above, except with
	\begin{equation} \label{eq:separation.control.contradiction}
		j \cdot d(\cR_{\geq \rho}(T_j) \cap \partial B_1(\bOh), \spt T_j') < d(\cR_{\geq A \rho}(T_j) \cap \partial B_A(\bOh), \spt T_j').
	\end{equation}
	Since the right hand side is uniformly bounded from above, we can pass to a subsequence (not denoted) along which $T_j, T_j' \to T$, a minimizing boundary with $\bOh \in \spt T$. 
	
	For each $j$, let $\bx_j \in \cR_{\geq \rho}(T_j)$ be the point on $\spt T_j$ attaining the distance on the left hand side of \eqref{eq:separation.control.contradiction}. Passing to a further subsequence (not labeled), $\bx_j \to \bx \in \cR_{\geq \rho}(T)$ by Remark \ref{remark:reg.scale}. Then, by renormalizing by the right hand side of \eqref{eq:separation.control.contradiction}, we obtain a nonnegative Jacobi field on $\reg T$ that equals zero at $\bx$. 
	
	On the other hand, by Lemma \ref{lemm:nontrivial.reg}, $\reg T$ also contains points in $\cR_{\geq A \rho}(\cdot) \cap \partial B_A(\bOh)$, and the limiting Jacobi field isn't everywhere zero on the component by \eqref{eq:separation.control.contradiction}. This contradicts the maximum principle.
\end{proof}

\begin{lemma} \label{lemm:jacobi.field.growth}
	For every nonflat minimizing cone $\cC$ in $\RR^{n+1}$, every positive Jacobi field $u$ on $\reg \cC$, every $r \in [1, \infty)$, and every $\rho \in (0, \rho_n)$ with $\rho_n$ as in Lemma \ref{lemm:nontrivial.reg}:
	\[ \sup_{\cR_{\geq r \rho}(\cC) \cap \partial B_r(\bOh)} u \leq H r^{-\kappa_n} \inf_{\cR_{\geq \rho}(\cC) \cap \partial B_1(\bOh)} u, \]
	where $H = H(n, \rho)$.
\end{lemma}
\begin{proof}
	This follows from our proof of \cite[Corollary 3.11]{CMS:generically.smooth.10}.
\end{proof}

\begin{lemma} \label{lemm:conical.decay}
	Let $\lambda \in (0, \kappa_n+1)$, $\rho \in (0, \rho_n)$ be given, with $\rho_n$ as in Lemma \ref{lemm:nontrivial.reg}. There exist $\delta = \delta(n, \lambda, \rho) \in (0, \tfrac12)$, $A = A(n, \lambda, \rho) \in (1, (2\delta)^{-1})$ with the following property.
	
	Consider any minimizing boundary $T$ in $B_{\delta^{-1}}(\bOh)$, with $\bOh \in \sing T$ satisfying
	\[ \Theta_T(\bOh, 1) \geq \Theta_T(\bOh, 2) - \delta. \]
	Then, for every minimizing boundary $T'$ in $B_{\delta^{-1}}(\bOh)$ not crossing $T$ smoothly, and with $\spt T' \cap \bar B_\delta(\bOh) \neq \emptyset$, we also have:
	\[ A^{-1} d(\cR_{\geq A \rho}(T) \cap \partial B_{A}(\bOh), \spt T') \leq A^{-\lambda} d(\cR_{\geq \rho}(T) \cap \partial B_1(\bOh), \spt T'), \]
	and all sets above are nonempty.
\end{lemma}
\begin{proof}
	Without loss of generality, we may suppose $T$, $T'$ have connected supports. 
	
	First we choose $A = A(n, \lambda, \rho)$ sufficiently large so that
	\begin{equation} \label{eq:conical.decay.A}
		H A^{\lambda} \leq \tfrac12 A^{1+\kappa_n}
	\end{equation}
	with $H = H(n, \rho)$ as in Lemma \ref{lemm:jacobi.field.growth}.
	
	We argue by contradiction. By Lemma \ref{lemm:nontrivial.reg}, $\cR_{\geq \rho}(T) \cap \partial B_1(\bOh)$, $\cR_{\geq A \rho}(T) \cap \partial B_{A}(\bOh)$ are both nonempty for $\delta \leq \tfrac12$. So let's assume that $T_j$ and $T_j'$ are as above with $\delta = j^{-1}$ and $j$ large enough that $j > 2 A$, and
	\begin{equation} \label{eq:conical.decay.density}
		\Theta_{T_j}(\bOh, 2) - \Theta_{T_j}(\bOh, 1) \leq j^{-1}, 
	\end{equation}
	\begin{equation} \label{eq:conical.decay.distance}
		\spt T_j' \cap \bar B_{j^{-1}}(\bOh) \neq \emptyset,
	\end{equation}
	but 
	\begin{equation} \label{eq:conical.decay.contradiction.0}
		A^{-\lambda}d(\cR_{\geq \rho}(T_j) \cap \partial B_1(\bOh), \spt T_j') < A^{-1} d(\cR_{\geq A \rho}(T_j) \cap \partial B_{A}(\bOh), \spt T_j').
	\end{equation}
	
	Note that \eqref{eq:conical.decay.density} implies that, after perhaps passing to a subsequence (not labeled), $T_j \to \cC$, a nonflat minimizing cone $\cC$. Then, \eqref{eq:conical.decay.distance}, the strong maximum principle, and the connectedness of $\reg \cC$, imply that $T_j' \to \cC$ as well. 
	
	One may now construct a positive Jacobi field on $\reg \cC$ that reflects \eqref{eq:conical.decay.contradiction.0}. Since this construction is standard, we will omit the technical details and refer the reader to the derivation of \cite[(10)]{Simon:smp} on \cite[p. 333]{Simon:smp}. Fix some arbitrary open $U \Subset \reg \cC$, which we may take to be connected since $\reg \cC$ is. Since $T_j, T_j'$ converge locally smoothly to $\cC$ away from $\sing \cC$, the height functions $h_j, h_j'$ of $\reg T_j, \reg T_j'$ over $U$ satisfy $h_j, h_j' \to 0$ smoothly on $U$. Moreover, $u_j = h_j - h_j'$ has a fixed sign since $T_j, T_j'$ do not cross smoothly. It is not hard to see that $u_j$ satisfies an elliptic equation of the form
	\[ \Delta_{\cC} u_j + |A_{\cC}|^2 u_j = \Div_{\cC}(a_j \cdot \nabla_{\cC} u_j) + b_j \cdot \nabla_{\cC} u_j + c_j u_j \text{ on } U, \]
	where $a_j, b_j, c_j \to 0$ smoothly on $U$. Now the connectedness of $U$ and the standard Harnack inequality for divergence-form elliptic equations allows us to renormalize $u_j$ and, after passing to a subsequence (not labeled), obtain a positive Jacobi field, i.e., a solution $u > 0$ of
	\[ \Delta_{\cC} u + |A_{\cC}|^2 u = 0 \text{ on } U. \]
	At this point we may apply this process with an exhaustion of $\reg \cC$ by such precompact $U$'s and have $u$ be defined over all of $\reg \cC$.
	
	Next using the fact that the vertical distance is within $o(1)$ of the distance in \eqref{eq:conical.decay.contradiction.0} over the subsets $\cR_{\geq \rho}(\cC)$ of controlled curvature, we obtain, using Remark \ref{remark:reg.scale},
	\begin{equation} \label{eq:conical.decay.contradiction.1}
		A^{-\lambda} \inf_{\cR_{\geq \rho}(\cC) \cap \partial B_1(\bOh)} u \leq A^{-1} \sup_{\cR_{\geq A \rho}(\cC) \cap \partial B_{A}(\bOh)} u;
	\end{equation}
	where $u$ is the positive Jacobi field constructed on $\reg \cC$. By Lemma \ref{lemm:jacobi.field.growth} with $r=A$, \eqref{eq:conical.decay.contradiction.1} implies
	\begin{equation} \label{eq:conical.decay.contradiction.2}
		A^{-\lambda} \inf_{\cR_{\geq \rho}(\cC) \cap \partial B_1(\bOh)} u \leq H A^{-1-\kappa_n} \inf_{\cR_{\geq \rho}(\cC) \cap \partial B_1(\bOh)} u.
	\end{equation}
	After canceling out the common term from both sides, \eqref{eq:conical.decay.contradiction.2} contradicts \eqref{eq:conical.decay.A}.
\end{proof}

We now come to the main proof of this section.

\begin{proof}[Proof of Theorem \ref{theo:timestamp.regularity}]
	Let $\alpha \in (0, \kappa_n + 1)$. Fix $\rho \in (0, \rho_n)$, $\lambda \in (\alpha, \kappa_n+1)$. Then let $\delta = \delta(n, \lambda, \rho)$ and $A = A(n, \lambda, \rho)$ be as in Lemma \ref{lemm:conical.decay}, and $L = L(n, A, \rho) = L(n, \lambda, \rho)$ be as in Lemma \ref{lemm:separation.control}. 
	
	Using the compactness of $\mathscr{F}$ and the upper semicontinuity of density, there exists $\Theta \in (1, \infty)$ such that
	\begin{equation} \label{eq:timestamp.regularity.theta}
		 \Theta_{T}(\by) \leq \Theta \text{ for all } T \in \mathscr{F}, \; \by \in \sing T.
	\end{equation}
	Using assumption (b), there exists $r > 0$ such that 
	\begin{equation} \label{eq:timestamp.regularity.r.1}
		T \mres B_{2r}(\by) \text{ is a minimizing boundary for all } T \in \mathscr{F}, \; \by \in \sing \mathscr{F}.
	\end{equation}
	(we are not necessarily assuming that $\by \in \sing T$) and, again by the compactness of $\mathscr{F}$,
	\begin{equation} \label{eq:timestamp.regularity.r.2}
		\Theta_{T}(\by, 2r) \leq 2 \Theta \text{ for all } T \in \mathscr{F}.
	\end{equation}
	Then, let $\gamma \in (0, \delta)$ be such that
	\begin{equation} \label{eq:timestamp.regularity.gamma.1}
		2 \gamma^{\lambda-\alpha} < 1.
	\end{equation}
	
	\begin{claim} \label{clai:timestamp.regularity}
		For sufficiently large $m \in \NN$, we have for all $\by \in \sing T_s$, $\by' \in \spt T_{s'}$,
		\[ |\by' - \by| < \gamma^m \implies |s'-s| < 2^m \gamma^{m\lambda}. \]
	\end{claim}
	\begin{proof}
		Suppose not. Then, perhaps after passing to a subsequence of $m$'s, there would exist $\by_m \in \sing T_{s_m}$, $\by_m' \in \spt T_{s_m'}$ violating the estimate, i.e., so that
		\begin{equation} \label{eq:timestamp.regularity.contradiction.1}
			|\by_m' - \by_m| < \gamma^m,
		\end{equation}
		\begin{equation} \label{eq:timestamp.regularity.contradiction.2}
			|s_m' - s_m| \geq 2^m \gamma^{m \lambda}.
		\end{equation}
		We may assume $m$ is large enough that
		\begin{equation} \label{eq:timestamp.regularity.contradiction.3}
			\gamma^{m-1} < r \delta.
		\end{equation}
		For each $q = 0, 1, 2, \ldots$, define
		\[ T_{m,q} = (\eta_{\by_m,A^q \gamma^{m-1}})_{\#} T_{s_m}, \]
		\[ T'_{m,q} = (\eta_{\by_m,A^q \gamma^{m-1}})_{\#} T_{s_m'}. \]
		\[ \bx_{m,q}' = (\eta_{\by_m, A^q \gamma^{m-1}})_{\#}(\by_m'). \]
		Observe that
		\begin{equation} \label{eq:timestamp.regularity.induction}
			T_{m,q+1} = (\eta_{\bOh,A})_{\#} T_{m,q}, \; T'_{m,q+1} = (\eta_{\bOh,A})_{\#} T'_{m,q},
		\end{equation}
		and, using $A > 1$ and \eqref{eq:timestamp.regularity.contradiction.1},
		\begin{equation} \label{eq:timestamp.regularity.small.dist}
			\Vert \bx_{m,q}' \Vert = A^{-q} \gamma^{1-m} \Vert \by_m' - \by_m \Vert < \gamma \implies \spt T_{m,q}' \cap B_\gamma(\bOh) \neq \emptyset.
		\end{equation}
		
		Let $Q$ be the largest integer satisfying $A^Q \gamma^{m-1} < r \delta$, i.e., 
		\begin{equation} \label{eq:timestamp.regularity.Q}
			A^Q \gamma^{m-1} < r \delta \leq A^{Q+1} \gamma^{m-1}.
		\end{equation}
		For all $q = 0, \ldots, Q$, $T_{m,q}$, $T_{m,q}'$ are minimizing boundaries in $B_{\delta^{-1}}(\bOh)$ by \eqref{eq:timestamp.regularity.contradiction.3}, and not smoothly crossing by (a). Moreover, for $q = 0, \ldots, Q-1$, there are two mutually exclusive possibilities:
		\begin{enumerate}
			\item[(A)] $\Theta_{T_{m,q}}(\bOh, A) - \Theta_{T_{m,q}}(\bOh, 1) < \delta$. Then by Lemma \ref{lemm:conical.decay}, \eqref{eq:timestamp.regularity.induction}, and \eqref{eq:timestamp.regularity.small.dist},
			\begin{align*}
				& d(\cR_{\geq \rho}(T_{m,q+1}) \cap \partial B_1(\bOh), \spt T_{m,q+1}') \\
				& \qquad = A^{-1} d(\cR_{\geq A \rho}(T_{m,q}) \cap \partial B_{A}(\bOh), \spt T_{m,q}') \\
				& \qquad \leq A^{-\lambda} d(\cR_{\geq \rho}(T_{m,q}) \cap \partial B_1(\bOh), \spt T_{m,q}'), 
			\end{align*}
			\item[(B)] $\Theta_{T_m}(\bOh, A) - \Theta_{T_m}(\bOh, 1) \geq \delta$. Then by Lemma \ref{lemm:separation.control} and \eqref{eq:timestamp.regularity.induction},
			\begin{align*}
				& d(\cR_{\geq \rho}(T_{m,q+1}) \cap \partial B_1(\bOh), \spt T_{m,q+1}') \\
				& \qquad = A^{-1} d(\cR_{\geq A \rho}(T_{m,q}) \cap \partial B_{A}(\bOh), \spt T_{m,q}') \\
				& \qquad \leq (L/A) d(\cR_{\geq \rho}(T_{m,q}) \cap \partial B_1(\bOh), \spt T_{m,q}').
			\end{align*}
		\end{enumerate}
		Let $Q_A, Q_B$ denote the number of times that possibilities (A), (B) occur, respectively. Obviously, $Q_A + Q_B = Q$, and by the monotonicity formula together with \eqref{eq:timestamp.regularity.r.1}, \eqref{eq:timestamp.regularity.r.2}, we also have that $Q_B \leq 2 \Theta \delta^{-1}$. In particular, by \eqref{eq:timestamp.regularity.induction} and the crude initial estimate
		\[ d(\cR_{\geq \rho}(T_{m,0}) \cap \partial B_1(\bOh), \spt T_{m,0}') \leq 2 \]
		we deduce after $Q$ iterations that
		\begin{align*}
			& A^{-Q} \gamma^{1-m} d(\cR_{\geq A^Q \gamma^{m-1} \rho}(T_{s_m}) \cap \partial B_{A^Q \gamma^{m-1}}(\by_m), \spt T_{s_m'}) \\
			& \qquad = d(\cR_{\geq \rho}(T_{m,Q}) \cap \partial B_1(\bOh), \spt T_{m,Q}') \\
			& \qquad \leq (A^{-\lambda})^{Q_A} (L/A)^{Q_B} d(\cR_{\geq \rho}(T_{m,0}) \cap \partial B_1(\bOh), \spt T_{m,0}') \\
			&\qquad = (A^{-\lambda})^{Q} (LA^{\lambda-1})^{Q_B} d(\cR_{\geq \rho}(T_{m,0}) \cap \partial B_1(\bOh), \spt T_{m,0}') \\
		& \qquad \leq 2 (A^{-\lambda})^{Q} (\max\{LA^{\lambda-1},1\})^{2 \Theta \delta^{-1}}.
		\end{align*}
		Then, using \eqref{eq:timestamp.regularity.Q} we deduce
		\begin{equation} \label{eq:timestamp.regularity.interior.small}
			d(\cR_{\geq A^{-1} r \delta \rho}(T_{s_m}) \cap \bar B_{r \delta}(\by_m), \spt T_{s_m'}) \leq L' \gamma^{m\lambda},
		\end{equation}
		where $L' = L'(A, \gamma, \delta, \Theta, r)$. Together \eqref{eq:timestamp.regularity.contradiction.2} and \eqref{eq:timestamp.regularity.interior.small} are in contradiction since they imply the existence of a nonnegative Jacobi field on $\lim_m T_{s_m}$ (this exists after passing to a subsequence) with an interior vanishing point but which is positive on the boundary by assumptions (c), (d).
	\end{proof}
	
	It follows from the claim that for large $m \in \NN$ and all $\bx \in \sing T_s$, $\bx' \in \spt T_{s'}$, 
	\begin{align*}
		\gamma^{m+1} \leq |\bx' - \bx| < \gamma^m 
			& \implies |\mathfrak{t}(\bx') - \mathfrak{t}(\bx)| < 2^m \gamma^{m\lambda} \\
			& \implies \frac{|\mathfrak{t}(\bx') - \mathfrak{t}(\bx)|}{|\bx'-\bx|^\alpha} < 2^m \gamma^{m\lambda} \gamma^{-(m+1)\alpha} = \gamma^{-\alpha} (2 \gamma^{\lambda-\alpha})^m,
	\end{align*}
	so,  in view of \eqref{eq:timestamp.regularity.gamma.1}, $\mathfrak{t}$ is indeed $\alpha$-H\"older on the singular set.
\end{proof}

\section{Proof of Corollary \ref{coro:main}}\label{sec:proof-main-cor}

Apply Theorem \ref{theo:singular.set.size} to small balls locally away from $\cup_{s \in [-\delta, \delta]} \Gamma_s$'s, small enough that each $T_s \in \mathscr{M}(\Gamma_s)$ restricts to a minimizing boundary in the ball. Then taking countable unions we deduce that
\[ \dim_H \cS^\ell(\mathscr{F}) \leq \ell \text{ for all } \ell \in \NN. \]
Note that, by Theorem \ref{theo:timestamp.regularity}, the measure theoretic result in \cite[Proposition 7.7 (a)]{FigalliRosOtonSerra:obstacle.generic.regularity} applies to $\cS^\ell(\mathscr{F})$ with $\ell = 0, 1, 2$ (since $2 < 2 + \eps_n$) and yields a full-measure subset
\[ I_{\ell} \subset [-\delta, \delta], \; \ell = 0, 1, 2, \]
with the following property:
\[ \ell = 0, 1, 2, \; s \in I_{\ell}, \; T_s \in \mathscr{T}(\Gamma_s) \implies \cS^\ell(T_s) = \emptyset. \]
Likewise, \cite[Proposition 7.7 (b)]{FigalliRosOtonSerra:obstacle.generic.regularity} applies to $\cS^\ell(\mathscr{F})$ with $\ell \geq 3$ (since $3 \geq 2 + \eps_n$) and yields a full-measure subset
\[ I_{\ell} \subset [-\delta, \delta], \; \ell \geq 3, \]
with the following property:
\[ \ell \geq 3, \; s \in I_{\ell}, \; T_s \in \mathscr{M}(\Gamma_s) \implies \dim_H \cS^\ell(T_s) \leq \ell - 2 - \eps_n. \]
The result follows since the intersection $\cap_\ell I_\ell$ remains a full-measure subset of $[-\delta, \delta]$.

\section{Proof of Theorem \ref{theo:main.geometric}}\label{theo:construct-foliation}

Given all our tools so far, the strategy is straightforward: we would like to construct a family of boundary perturbations $(\Gamma_s)_{s \in [-\delta, \delta]}$ of $\Gamma$ on which to apply Corollary \ref{coro:main}. The two main difficulties are the potential non-uniqueness of $T$ among minimiers, and the possible presence of high multiplicity on $T$. 

For simplicity, we break down the proof into steps.

\subsection{Reduction to uniquely minimizing $T$} \label{subsec:main.geometric.unique.replacement}

It follows from the Hardt--Simon boundary regularity theorem (\cite[Corollary 11.2]{HardtSimon:boundary.regularity}) that 
\begin{equation} \label{eq:main.hs.1}
	\spt T = \bar M
\end{equation} for an oriented hypersurface $M$ with nonempty boundary, which satisfies
\begin{equation} \label{eq:main.hs.2}
	\partial M \subset \Gamma, \; \sing T = \bar M \setminus M \subset \RR^{n+1} \setminus \Gamma.
\end{equation}
In our previous paper we showed that perturbing $\Gamma \mapsto \Gamma'$ by pushing the components of $\partial M \subset \Gamma$ inward along $M$ forces $\mathscr{M}(\Gamma')$ to be a singleton; see \cite[Lemma A.3]{CMS:generically.smooth.10}. So without loss of generality and after relabeling $\Gamma' \mapsto \Gamma$ we may assume that
\begin{equation} \label{eq:main.unique}
	\mathscr{M}(\Gamma) = \{ T \} \; (\iff T \text{ is uniquely minimizing}).
\end{equation}
Note that \eqref{eq:main.unique} and the compactness theorem for integral $n$-currents combine to yield
\begin{equation} \label{eq:main.unique.nearby.weakly.close}
	\Gamma' \to \Gamma \text{ smoothly and } \; T' \in \mathscr{M}(\Gamma') \implies T' \to T.
\end{equation}

\subsection{The case of $T$ with multiplicity one} \label{subsec:main.geometric.mult.one}

In this case, the Hardt--Simon boundary regularity theorem (\cite[Corollary 11.2]{HardtSimon:boundary.regularity}) further guarantees that
\begin{equation} \label{eq:main.hs.3}
	T = \llbracket M \rrbracket, \; \partial M = \Gamma,
\end{equation}
for the same hypersurface $M$ that satisfies \eqref{eq:main.hs.1}, \eqref{eq:main.hs.2}.

Next, by the upper semicontinuity of density, together with Remark \ref{remark:reg.scale} and \eqref{eq:main.unique.nearby.weakly.close}, the fact that $T$ has multiplicity one also implies
\begin{equation} \label{eq:main.unique.nearby.mult.one}
	\Gamma' \to \Gamma \text{ smoothly and } \; T' \in \mathscr{M}(\Gamma') \implies T' \text{ also has multiplicity one}.
\end{equation}
Therefore, all such $T'$ themselves have decompositions satisfying \eqref{eq:main.hs.1}, \eqref{eq:main.hs.2}, \eqref{eq:main.hs.3} with $\Gamma'$ in place of $\Gamma$, $T'$ in place of $T$, and $M'$ in place of $M$. 

It follows from \eqref{eq:main.unique.nearby.weakly.close}, \eqref{eq:main.unique.nearby.mult.one}, and Allard's interior (\cite[\S 5]{Simon:GMT}) and boundary (\cite{Allard:varifold.boundary}) regularity theorems (note that in the multiplicity-one case $T$ has density $\tfrac12$ on $\Gamma$) that
\begin{equation} \label{eq:main.unique.nearby.strongly.close}
	\Gamma' \to \Gamma \text{ smoothly and } \; T' = \llbracket M' \rrbracket \in \mathscr{M}(\Gamma') \implies M' \to M \text{ locally smoothly}
\end{equation}
in the sense of smooth embeddings; by ``locally smoothly,'' we mean the convergence is smooth on compact subsets of $M$ (including up to $\Gamma = \partial M$).  

To proceed further we will need to restrict to graphical perturbations $\Gamma'$ of $\Gamma$. Fix  $\Gamma$ and a (incomplete) hypersurface $\Sigma$ with boundary, such that
\[ \partial \Sigma =  \Gamma, \; \bar \Sigma \subset M, \]
respecting orientations (e.g., $\Sigma = M \cap U$ for a small tubular neighborhood $U$ of $\Gamma$).
Now let $\delta > 0$ be small enough that each
\[ \Sigma_s := \Graph_\Sigma s, \; s \in [-\delta, \delta] \]
is still a smooth hypersurface with boundary, and denote
\[ \Gamma_s := \partial \Sigma_s\]
so that $\Sigma_0 = \Sigma$, $\Gamma_0 = \Gamma$. Below we will only need the $\Gamma_s$, and may discard the $\Sigma_s$.

\begin{claim} \label{clai:main.mult.one.conditions}
	After possibly shrinking $\delta > 0$, the family $(\Gamma_s)_{s \in [-\delta, \delta]}$ satisfies the assumptions of Theorem \ref{theo:timestamp.regularity}.
\end{claim}

Given Claim \ref{clai:main.mult.one.conditions}, Corollary \ref{coro:main} applies to $(\Gamma_s)_{s \in [-\delta, \delta]}$. Thus, for a.e.\ $s \in [-\delta, \delta]$, every $T_s \in \mathscr{M}(\Gamma_s)$ has the desired improved regularity of Theorem \ref{theo:main.geometric}. So it remains to prove Claim \ref{clai:main.mult.one.conditions}.

\begin{proof}[Proof of Claim \ref{clai:main.mult.one.conditions}]
	By inspecting Corollary \ref{coro:main} we see that we need to verify conditions (a), (b), (c), (d) in Theorem \ref{theo:timestamp.regularity}:
	\begin{enumerate}
		\item[(a)] This holds by the well-known cut-and-paste technique for minimizers; see, e.g., \cite[Lemma 2.8]{CMS:generically.smooth.10}.
		\item[(b)] This holds, after perhaps shrinking $\delta$, by \eqref{eq:main.unique.nearby.mult.one}.
		\item[(c)] This holds, after perhaps shrinking $\delta$, by \eqref{eq:main.unique.nearby.strongly.close}.
		\item[(d)] This holds automatically with $\alpha = 1$.
	\end{enumerate}
	This completes the proof of the claim.
\end{proof}

\subsection{The general case} \label{subsec:main.geometric.high.mult}

It follows from the Hardt--Simon boundary regularity theorem (\cite[Corollary 11.2]{HardtSimon:boundary.regularity}) and its refinement by White (\cite[Corollary 2]{White:boundary.regularity.multiplicity}) that
\[
	T \in \mathscr{M}(\Gamma) \implies T = T_{1} + \ldots + T_{m}
\]
with each $T_i$ being a multiplicity-one minimizer satisfying \eqref{eq:main.hs.1}, \eqref{eq:main.hs.2}, \eqref{eq:main.hs.3} with $T_i$ in place of $T$, $M_i$ in place of $M$, some union of components $\Gamma_i \subset \Gamma$ in place of $\Gamma$, and
\begin{equation} \label{eq:main.high.mult.nested}
	\bar M_{j} \subset \bar M_{i} \setminus \Gamma_{i} \text{ for all } i < j;
\end{equation}
see also \cite[Theorem A.1]{CMS:generically.smooth.10}. Note that $m$ equals the largest multiplicity of $T$ on $\reg T$. We'll refer to this as the (Hardt--Simon) ``decomposition'' of $T$.

The decomposition above applies to any minimizing $T' \in \mathscr{M}(\Gamma')$ in place of $T$, with $\Gamma'$ in place of $\Gamma$, $M'_{i}$ in place of $M_{i}$, $\Gamma'_{i}$ in place of $\Gamma_{i}$, and $m'$ in place of $m$. 

\begin{claim} \label{clai:main.high.mult.weakly.close}
	If $\Gamma'$ is sufficiently close to $\Gamma$ in $C^\infty$, and $T' \in \mathscr{M}(\Gamma')$ is arbitrary, then the decomposition for $T'$ has $m' \leq m$. In fact, either
	\begin{enumerate}
		\item[(a)] $m' < m$, or
		\item[(b)] $m' = m$ and $M'_m$ has strictly fewer components than $M_m$, or
		\item[(c)] $m' = m$ and $\llbracket M'_m \rrbracket$ is close to $\llbracket M_m \rrbracket$.
	\end{enumerate}
\end{claim}
\begin{proof}
	Throughout, we'll implicitly use \eqref{eq:main.unique.nearby.weakly.close}, the fact that components of any $M_i$ are in bijection with components of $\bar M_i$ (see \cite[Lemma 2.5]{CMS:generically.smooth.10}), and the characterization of $m$ as the top multiplicity of $T$, and respectively all the same statements for $T'$.
	
	It follows from the upper semicontinuity of density that $m' \leq m$. So, going forward we may assume that (a) fails (otherwise we're done), and thus $m' = m$. 
	
	By our decomposition, $T$ has multiplicity $\leq m-1$ on the complement of $M_m$ and $\bar M_m \cap \Gamma = \partial M_m$, so the upper semicontinuity of densities also yields that $M'_m$ converges to a subset of $M_m$. Note that distinct components of $M'_m$ cannot limit to subsets of the same component of $M_m$ (this follows as in (b) in \cite[Lemma 4.6]{CMS:generically.smooth.10}), so $M'_m$ has at most as many components as $M_m$. Going forward, we may suppose that (b) fails too (otherwise we're done). Then, $M_m$ and $M'_m$ have the same number of components. 
	
	By Allard's theorem \cite[\S 5]{Simon:GMT} and the interior regularity of minimizers \cite{DeGiorgi:regularity}, it follows that away from $\Gamma$, $T'$ decomposes as a multisheeted graph (the sheets pairwise don't intersect, or they overlap) locally over $M_m \setminus \partial M_m$ with sheet multiplicities equal to the density of $T'$. Since $m' = m$ and $M_m$, $M'_m$ have the same number of components, it follows $T'$ is a single graph with multiplicity $m$ locally over $M_m \setminus \partial M_m$. Thus, $\partial M'_m \to \partial M_m$. This proves (c).
\end{proof}

\begin{claim} \label{clai:main.high.mult.reduction}
	If $T$ is not of multiplicity one (i.e., $m \geq 2$), then there exist $\Gamma' \to \Gamma$ so that each $T' \in \mathscr{M}(\Gamma')$ satisfies (a) or (b) in Claim \ref{clai:main.high.mult.weakly.close}.
\end{claim}
\begin{proof}
	Without loss of generality, we may assume that $\spt T$ is connected. 
	
	Perturb $\Gamma \mapsto \Gamma'$ so that $\Gamma_m$ gets pushed off $\spt T$, while all other components of $\Gamma$ stay fixed. Now suppose, for contradiction, that $\tilde T' = T' \in \mathscr{M}(\Gamma')$ satisfies (c). 
	
	It follows that $T' - \llbracket M_m' \rrbracket$ is a minimizer with prescribed boundary $\sum_{i=1}^{m-1} \llbracket \Gamma_i \rrbracket$. This is the same as the boundary of $\tilde T = T - \llbracket M_m \rrbracket$, so $\tilde T = \tilde T'$ by \S \ref{subsec:main.geometric.unique.replacement} (otherwise we'd get a nonunique minimizer for $\Gamma$), so $\spt T = \spt T'$ by \eqref{eq:main.high.mult.nested}, a contradiction.
\end{proof}

Note that we can repeatedly invoke Claim \ref{clai:main.high.mult.reduction} and \S \ref{subsec:main.geometric.unique.replacement}, replacing $T' \mapsto T$ at the end of each step, until $T$ is of multiplicity one, in which case the result follows from \S \ref{subsec:main.geometric.mult.one}.


\end{document}